\documentclass{amsart}
\usepackage[utf8]{inputenc}
\usepackage{amsmath,soul}
\usepackage{amsthm}
\usepackage{amsfonts}
\usepackage{amssymb}
\usepackage{enumerate}
\usepackage{tikz}
\usepackage{float}
\usepackage{wrapfig}
\usepackage[foot]{amsaddr}
\usepackage{pgfplots}
\pgfplotsset{compat=1.9}
\usepackage[backend=biber, style=numeric-comp]{biblatex}
\usepackage{caption}
\allowdisplaybreaks
\newcommand{\met}{d}
\newcommand{\metset}{\mathcal{M}}
\theoremstyle{plain}
\newtheorem{theorem}{Theorem}
\newtheorem{lemma}{Lemma}
\newtheorem{proposition}{Proposition}
\newtheorem{cor}{Corollary}
\theoremstyle{definition}
\newtheorem{definition}{Definition}
\theoremstyle{remark}
\newtheorem{remark}{Remark}
\newtheorem{example}{Example}
\title{Common best proximity point theorems under proximal $F$-weak dominance in complete metric spaces}
\author[1]{Aman Deep}
\address[1]{Department of Mathematics, University of Delhi, Delhi-110007}
\email{amansharmacpl001@gmail.com}

\author[2]{Rakesh Batra}
\email{rakeshbatra.30@gmail.com}
\address[2]{Department of Mathematics, Hansraj College, University of Delhi, Delhi-110007}

\subjclass[2020]{54H25, 47H10, 55M20}
\keywords{Best proximity point, $F$-dominance, $F$-weak dominance, proximal $F$-weak dominance, fixed point}
\begin{document}
\begin{abstract}
    Suppose that $S_1$ and $S_2$ are nonempty subsets of a complete metric space $(\metset,\met)$ and $\phi,\psi:S_1\to S_2$ are mappings. The aim of this work is to investigate some conditions on $\phi$ and $\psi$ such that the two functions, one that assigns to each $x\in S_1$ exactly $\met(x,\phi x)$ and the other that assigns to each $x\in S_1$ exactly $\met(x,\psi x)$, attain the global minimum value at the same point in $S_1$. We have introduced the notion of proximally $F$-weakly dominated pair of mappings and proved two theorems that guarantee the existence of such a point. Our work is an improvement of earlier work in this direction. We have also provided examples in which our results are applicable, but the earlier results are not applicable.
\end{abstract}
\maketitle

%Best proximity point \sep $F$-dominance \sep $F$-weak dominance \sep  proximal $F$-weak dominance \sep fixed point.

%\MSC[2020] 54H25 \sep 47H10 \sep 55M20.
\section{Introduction}
After the discovery of contraction mapping theorem by Stefan Banach in 1922, it became apparent that it was a very useful result. This prompted other scholars towards finding similar results, and thus the vast field of metric fixed point theory was born. For a detailed discussion on such theorems, one may refer to \cite{handbook} and the references therein.

One of the recent generalizations of contraction mapping theorem was given by Wardowski \cite{wardowski2012fixed}, who defined the so called $F$-contraction mappings. He not only proved that every $F$-contraction satisfies the conclusion of contraction mapping theorem, but also produced an example of a mapping which is an $F$-contraction but does not satisfy the hypothesis of contraction mapping theorem. Many scholars tried to generalize the concept of an $F$-contraction in many directions. The recent survey by Karapinar \textit{et al.} \cite{karapinar2020survey} contains a detailed discussion of various generalizations of $F$-contractions.

Suppose that $S_1$ and $S_2$ are nonempty subsets of a metric space $(\metset,\met)$ and $\phi:S_1\to S_2$ be a mapping. If $S_1\cap S_2=\emptyset$ then we cannot talk about a fixed point of $\phi$. In such a case, we search for points in $S_1$ such that the error $\met(x,\phi x)$ is as small as possible. The initial breakthrough in this direction was obtained by Ky Fan in 1969, who proved the following theorem.
\begin{theorem}[\cite{fan1969extensions}]
\label{fan-theorem}
Let $S$ be a nonempty compact convex set in a locally convex real  topological vector space $\metset$. Let $\phi:S\to \metset$ be a continuous mapping. Then either $\phi$ has a fixed point in $S$ or there exists a point $x$ in $S$ and a continuous seminorm $p$ on $\metset$ such that
$$0<p(x-\phi x)=\inf_{y\in S}p(y-\phi x).$$
\end{theorem}
In the modern terminology, theorems which guarantee a conclusion similar to that of Theorem \ref{fan-theorem} are called best approximation theorems. On the other hand, the theorems which guarantee the existence of a point $x\in S_1$ such that $\met(x,\phi x)=\met(S_1,S_2)=\inf\{\met(s_1,s_2):s_1\in S_1,s_2\in S_2\}$ are called best proximity point theorems. Some early work in this direction was due to Anuradha and Veeramani \cite{anuradha2009proximal}, Saddiq Basha \cite{basha2011best} and  Eldred \textit{et al.} \cite{eldred2005proximal}.

Recently, R. Batra \cite{batracommon} defined the so called proximally $F$-dominated pair of mappings and proved a common best proximity point theorem for such mappings. In doing so, he generalized the earlier work of Saddiq Basha \cite{basha2013common}. The main objective of our paper is to considerably weaken the notion of proximally $F$-dominated pair of mappings and to prove common best proximity point theorems for such mappings.

\section{Preliminaries}
We shall use the symbols $\mathbb{N}$, $\mathbb{R}$ and $\mathbb{R_+}$ to denote the set of non-negative integers, the set of real numbers and the set of positive real numbers, respectively. Convergence of a sequence $\{x_i\}_{i\in\mathbb{N}}$ to $x$ as $i\to \infty $ in a metric space will often be denoted by $x_i\to x$. The value of a function $\phi$ at $x$ will be denoted by $\phi x$ instead of $\phi(x)$ when there is no possibility of confusion. If $S_1$ and $S_2$ are sets, then the expression $S_1\subset S_2$ implies that $S_1$ is a subset of $S_2$. If $\phi:S_1\to S_2$ and $\psi:S_3\to S_4$ are functions such that $S_2\subset S_3$, then, $\psi\circ\phi:S_1\to S_4$ denotes the composition of functions, defined by $\psi\circ\phi x=\psi(\phi x)$ for all $x\in S_1$. The composition of a function $\phi:S \to S$ with itself $n$ times will be denoted by $\phi^n$, with the understanding that $\phi^0$ is the identity mapping of $S$. If $\phi:S_1\to\ S_2$ is a function, and $S\subset S_1$ then we denote by $\phi_{|S}$ the mapping $\phi_{|S}:S\to \phi(S)$ defined by $\phi_{|S}x=\phi x$ for all $x\in S$.

Let $S_1$ and $S_2$ be nonempty subsets of a metric space $(\metset,\met)$. Following S. Basha \cite{basha2013common}, we shall use the notations: $\met(S_1,S_2)$ as the infimum of the set of all possible distances between two points, one each from $S_1$ and $S_2$; and $S_1^0$(respectively, $S_2^0$) as the set of all $s_1\in S_1$(respectively, $s_2\in S_2$) such that $\met(S_1,S_2)$ is attained by $d$ at $(s_1,s_2)$ for some $s_2\in S_2$(respectively, $s_1\in S_1$). 

Now, we recall a well known notion of a closed mapping.
\begin{definition}[\cite{aliprantis} p. 51]
Let $(\metset_1,\met_1)$ and $(\metset_2,\met_2)$ be metric spaces and $\phi:\metset_1\to \metset_2$ be a mapping. We say that $\phi$ is closed if for every sequence $\{x_i\}_{i\in\mathbb{N}}$ in $\metset_1$,
$$[x_i\to x\in \metset_1 \text{ and } \phi x_i\to y\in \metset_2]\Rightarrow \phi x=y.$$
\end{definition}
\begin{remark}[\cite{aliprantis} p. 51]
\label{continuous-are-closed}
It is well known that every continuous mapping is closed, but the converse is not true.
\end{remark}
The notion of proximal commutativity will play a prominent role throughout the paper.
\begin{definition}[\cite{basha2013common}]
Let $S_1$ and $S_2$ be nonempty subsets of a metric space $(\metset,\met)$. A pair of mappings $\phi,\psi:S_1\to S_2$ is said to be proximally commuting, if
$$\left [\met(a,\phi x)=\met(b,\psi x)=\met(S_1,S_2)\right ] \Rightarrow \phi b=\psi a$$ for all $a, b, x\in S_1$.
\end{definition}
It is easy to see that if $S_1=S_2$, then $\phi$ and $\psi$ are proximally commuting if and only if they are commuting. Therefore, proximal commutativity is a generalization of usual commutativity.

Now, we recall the definition of common best proximity point.
\begin{definition}[\cite{basha2013common}]
Given any pair of nonempty subsets $S_1$ and $S_2$ of a metric space $(\metset,\met)$ and a pair of mappings $\phi,\psi:S_1\to S_2$, a point $x^*\in S_1$ is called a common best proximity point of $\phi$ and $\psi$ if $\met(x^*,\phi x^*)=\met(S_1,S_2)$ and $\met(x^*,\psi x^*)=\met(S_1,S_2)$.
\end{definition}
Since, the mappings $x\to \met(x,\phi x)$ and $x\to \met(x,\psi x)$ are generally nonlinear and $\met(S_1,S_2)\leq \met(x,\phi x)$ for all $x\in S_1$, hence, a common best proximity point problem may be regarded as a global nonlinear optimization problem.
In \cite{wardowski2012fixed}, Wardowski used a class of functions $F:\mathbb{R}_+\to\mathbb{R}$ satisfying the following properties to define the so called $F$-contractions.
\begin{enumerate}[({F}1)]
    \item $F$ is strictly increasing i.e. $x<y\Rightarrow Fx<Fy$ for all $x,y\in \mathbb{R}_+$;
    \item for every sequence $\{\alpha_i\}_{i\in\mathbb{N}}$ of positive real numbers, $\alpha_i\to 0\iff F\alpha_i\to -\infty$;
    \item there exists $k\in (0,1)$ such that $\lim_{\alpha\to 0^+}\alpha^kF\alpha=0.$
\end{enumerate}
We denote by $\mathcal{F}$ the class of all mappings $F:\mathbb{R}_+\to\mathbb{R}$ satisfying the conditions (F1)-(F3).
\begin{example}[\cite{wardowski2012fixed}]
\label{F-examples}
Consider the mappings $F_i:\mathbb{R}_+\to\mathbb{R}(i=1,2,3,4)$ defined by
\begin{enumerate}[(i)]
    \item $F_1\alpha=\ln{\alpha}$;
    \item $F_2\alpha=\ln{\alpha}+\alpha$;
    \item $F_3\alpha=\frac{-1}{\sqrt{\alpha}}$;
    \item $F_4\alpha=\ln(\alpha^2+\alpha).$
\end{enumerate}
It was ensured by Wardowski that $F_1,F_2,F_3,F_4\in\mathcal{F}$.
\end{example}
R. Batra \cite{batracommon} defined the concept of proximally $F$-dominated pair of mappings as follows:
\begin{definition}[\cite{batracommon}]
\label{F-dominating}
Given any $F\in\mathcal{F}$ and two nonempty subsets $S_1$ and $S_2$ of a metric space $(\metset,\met)$, a function $\psi:S_1\to S_2$ is said to $F$-dominate another function $\phi:S_1\to S_2$ proximally, if there exists a real number $\tau>0$ such that
\begin{multline*}
\left. \begin{array}{rll}
      \met(a_1,\phi x_1)=\met(a_2,\phi x_2)&=\met(S_1,S_2)\\
      \met(b_1,\psi x_1)=\met(b_2,\psi x_2)&=\met(S_1,S_2)\\
       &a_1\neq a_2
\end{array}\right\}\Rightarrow \left \{
\begin{array}{ll}
     &b_1\neq b_2\\
     & \tau+F(\met(a_1,a_2))\leq F(\met(b_1,b_2))
\end{array} \right.
\end{multline*}
for all $a_1,a_2,b_1,b_2,x_1,x_2\in S_1$.
\end{definition}
For all $a_1,a_2,b_1,b_2\in S_1$, let us define
$$M_\met(a_1,a_2,b_1,b_2)=\max\left\{\met(b_1,b_2),\met(a_2,b_2),\met(a_1,b_1),\frac{\met(a_2,b_1)+\met(a_1,b_2)}{2}\right\}.$$
  In \cite{wardowski2014weak}, Wardowski and Dung introduced the concept of $F$-weak contractions as follows:
\begin{definition}[\cite{wardowski2014weak}]
Let $(\metset,\met)$ be a metric space and $F\in\mathcal{F}$. Let $\phi:\metset\to \metset$. We say that $\phi$ is an $F$-weak contraction if there exists $\tau>0$ such that
\begin{equation*}
    \phi x_1\neq \phi x_2\Rightarrow \tau+F(\met(\phi x_1,\phi x_2))\leq F(M_\met(\phi x_1,\phi x_2,x_1,x_2))
\end{equation*}
for all $x_1,x_2\in \metset$.
\end{definition}
The following important generalization of contraction mapping theorem was also proved in \cite{wardowski2014weak}.
\begin{theorem}[\cite{wardowski2014weak}]
\label{F-weak-wardowski}
Let $F\in\mathcal{F}$ and $(\metset,\met)$ be a complete metric space. Let $\phi:\metset \to \metset$ be an $F$-weak contraction. If $\phi$ or $F$ is continuous, then $\phi$ has a unique fixed point $x^*$ in $\metset$. Moreover, for any $x_0\in \metset$, we have $\lim_{i\to\infty} \phi^ix_0=x^*$.
\end{theorem}
\section{Main Results}
Taking inspiration from \cite{wardowski2014weak}, we define the concept of proximally $F$-weakly dominated pair of mappings as follows:
\begin{definition}
\label{F-weakly-dominating}
Given any $F\in\mathcal{F}$ and two nonempty subsets $S_1$ and $S_2$ of a metric space $(\metset,\met)$, a function $\psi:S_1\to S_2$ is said to $F$-weakly dominate another function $\phi:S_1\to S_2$ proximally, if there exists a real number $\tau>0$ such that
\begin{multline*}
\left. \begin{array}{rll}
       \met(a_1,\phi x_1)=\met(a_2,\phi x_2)&=\met(S_1,S_2)\\
       \met(b_1,\psi x_1)=\met(b_2,\psi x_2)&=\met(S_1,S_2)\\
     &  a_1\neq a_2
\end{array}\right\}\\\Rightarrow
\left\{\begin{array}{ll}
     &   M_\met(a_1,a_2,b_1,b_2)\neq 0\\
     &  \tau+F(\met(a_1,a_2))\leq F(M_\met(a_1,a_2,b_1,b_2))
\end{array}\right.
\end{multline*}
for all $a_1,a_2,b_1,b_2,x_1,x_2\in S_1$.
\end{definition}
\begin{proposition}
\label{proper-extension-proposition}
For nonempty subsets $S_1$ and $S_2$ of a metric space $(\metset,\met)$ and $F\in\mathcal{F}$, if $\phi:S_1\to S_2$ is $F$-dominated by another mapping ${\psi:S_1\to S_2}$ proximally, then $\phi$ is also $F$-weakly dominated by $\psi$ proximally.
\end{proposition}
\begin{proof}
The proposition follows trivially from the fact that $F$ is strictly increasing and that $M_\met(a_1,a_2,b_1,b_2)\geq \met(b_1,b_2)$ for all ${a_1,a_2,b_1,b_2\in S_1}$.
\end{proof}
Example \ref{ex22} shows that the converse of Proposition \ref{proper-extension-proposition} may not be true. Hence, for a fixed pair of subsets $(S_1,S_2)$ of a metric space $(\metset,\met)$, the class of all pairs of mappings $(\phi,\psi)$ such that $\phi$ is $F$-dominated by $\psi$ proximally for some $F\in\mathcal{F}$ is generally a proper subclass of the class of all pairs of mappings $(\phi,\psi)$ such that $\phi$ is $F$-weakly dominated by $\psi$ proximally for some $F\in\mathcal{F}$.

If $S_1=S_2=\metset$, then the concept of best proximity point reduces to the concept of fixed point. The concept of $F$-weakly dominated pair of mappings proximally as introduced in Definition \ref{F-weakly-dominating} reduces to the following.
\begin{definition}
\label{f-g-x-propostion}
Let $(\metset,\met)$ be a metric space and $F\in\mathcal{F}$. Let $\phi$ and $\psi$ be functions mapping $\metset$ into $\metset$. Then, $\psi$ is said to $F$-weakly dominate $\phi$ if there exists $\tau>0$ such that
\begin{equation*}
\label{f-g-X-condition}
    \phi x_1\neq \phi x_2\Rightarrow \left\{\begin{array}{ll}
         &  M_\met(\phi x_1,\phi x_2,\psi x_1,\psi x_2)\neq 0\\
         &  \tau + F(\met(\phi x_1,\phi x_2))\leq F(M_\met(\phi x_1,\phi x_2,\psi x_1,\psi x_2))
    \end{array}\right.
\end{equation*}
for all $x_1,x_2\in \metset$.
\end{definition}
A definition similar to the Definition \ref{f-g-x-propostion} was first introduced by Zhou \textit{et al.} in \cite{zhou2019coincidence} under the name ``\'Ciri\'c-type $F_M$-contraction", however, the authors assume the continuity of $F$ in their definition. So, Definition \ref{f-g-x-propostion} is more general and is also an improvement of Definition 3.2 in \cite{batra2017common}, which is a special case of Definition \ref{F-dominating} with $S_1=S_2=\metset.$

The following proposition gives the relationship between $F$-weak contractions and $F$-weakly dominated pair of mappings.
\begin{proposition}
Let $(\metset,\met)$ be a metric space and $F\in\mathcal{F}$. Let $\phi:\metset\to \metset$ be a mapping and $I:\metset\to \metset$ be the identity mapping defined by $Ix=x$ for all $x\in \metset$. Then, $\phi$ is $F$-weakly dominated by $I$ if and only if $\phi$ is an $F$-weak contraction.
\end{proposition}
Therefore, the concept of $F$-weakly dominated pair of mappings can be considered a broad generalization of the concept of $F$-weak contractions.

Before proving the best proximity point theorems, we prove the following technical lemmas to avoid the repetition of arguments later.
\begin{lemma}
\label{sequence-lemma}
Let $F\in\mathcal{F}$ and $(\metset,\met)$ be a complete metric space. Consider any two nonempty subsets $S_1$ and $S_2$ of $\metset $ with $S_1^0$ as nonempty and closed. Let ${\phi,\psi:S_1\to S_2}$ be a pair of proximally commuting mappings which satisfy the following conditions:
\begin{enumerate}[(a)]
    \item $\psi$ is such that it $F$-weakly dominates $\phi$ proximally;
    \item $\phi(S_1^0)\subset S_2^0$;
    \item $\phi(S_1^0)\subset \psi(S_1^0)$
\end{enumerate}
then, there exists a sequence $\{a_i\}_{i\in\mathbb{N}}$ in $S_1^0$ which satisfies the following properties:
\begin{enumerate}
    \item $\phi a_{i}=\psi a_{i+1}$ for all $i\in\mathbb{N}$;
    \item if $a_i\neq a_{i+1}$ for all $i\in\mathbb{N}$, then there exists $a\in S_1^0$ such that $a_i\to a$.
\end{enumerate}
\end{lemma}
\begin{proof}
Choose $x_0\in S_1^0$ arbitrarily. Since $\phi(S_1^0)\subset \psi(S_1^0)$, we may define a sequence $\{x_i\}_{i\in\mathbb{N}}$ of points of $S_1^0$ inductively satisfying
\begin{equation}
    \label{x_i-relation}
    \phi x_{i}=\psi x_{i+1}
\end{equation}
for all $i\in\mathbb{N}$. Since $\phi(S_1^0)\subset S_2^0$, we can find a sequence $\{a_i\}_{i\in\mathbb{N}}$ in $S_1^0$ such that
\begin{equation}
\label{x_i-u_i-relation}
    \met(\phi x_i,a_i)=\met(S_1,S_2)
\end{equation}
for all $i\in\mathbb{N}$. The equations (\ref{x_i-relation}) and (\ref{x_i-u_i-relation}) give the following system of equations:
\begin{equation}
    \label{f-g-x_i-system}
    \left. \begin{array}{cc}
     &  \met(\phi x_{i+1},a_{i+1})=\met(S_1,S_2)\\
     &  \met(\psi x_{i+1},a_i)=\met(S_1,S_2)
\end{array}\right\}
\end{equation}
for all $i\in\mathbb{N}$.
The system (\ref{f-g-x_i-system}) along with the proximal commutativity of $\phi$ and $\psi$ gives
\begin{equation*}
    \phi a_i=\psi a_{i+1}
\end{equation*}
for all $i\in\mathbb{N}$.

Assume that $a_i\neq a_{i+1}$ for all $i\in\mathbb{N}$. We claim that $\{a_i\}_{i\in\mathbb{N}}$ is a Cauchy sequence. Suppose that $i\geq 1$. The equations (\ref{x_i-relation}) and (\ref{x_i-u_i-relation}) give
\begin{equation}
\label{main-relation}
    \left. \begin{array}{cc}
     &  \met(\phi x_{i+1},a_{i+1})=\met(\phi x_i,a_i)=\met(S_1,S_2)\\
     &  \met(\psi x_{i+1},a_i)=\met(\psi x_i,a_{i-1})=\met(S_1,S_2)\\
     &\quad\quad a_i\neq a_{i+1}.
\end{array}\right\}
\end{equation}
The system (\ref{main-relation}) along with the hypothesis (a) imply that there exists $\tau>0$ such that 
\begin{equation*}
    \begin{split}
        &F(\met(a_i,a_{i+1}))\leq F(\met(M_\met(a_{i+1},a_i,a_i,a_{i-1})))-\tau\\
        &=F\biggl(\max\biggl\{\met(a_i,a_{i-1}),\met(a_i,a_{i-1}),\met(a_{i+1},a_i),\\
        &\qquad\qquad\qquad\quad\quad\qquad\qquad\quad\left.\left. \frac{\met(a_i,a_i)+\met(a_{i+1},a_{i-1})}{2}\right\}\right)-\tau\\
        &=F\left(\max\left\{\met(a_i,a_{i-1}),\met(a_{i-1},a_i),\met(a_i,a_{i+1}),\frac{\met(a_{i+1},a_{i-1})}{2}\right\}\right)-\tau\\
        &=F\left(\max\left\{\met(a_i,a_{i-1}),\met(a_i,a_{i+1}),\frac{\met(a_{i+1},a_{i-1})}{2}\right\}\right)-\tau\\
        &\leq F\left(\max\left\{\met(a_i,a_{i-1}),\met(a_i,a_{i+1}),\frac{\met(a_i,a_{i+1})+\met(a_{i-1},a_i)}{2}\right\}\right)-\tau\\
        &=F\left(\max\left\{\met(a_i,a_{i+1}),\met(a_i,a_{i-1})\right\}\right)-\tau.
    \end{split}
\end{equation*}
In above calculations we took the advantage of the monotonicity of $F$ and the triangle inequality. Hence,
\begin{equation}
\label{big-manipulation}
F(\met(a_i,a_{i+1}))\leq F(\max\{\met(a_i,a_{i+1}),\met(a_i,a_{i-1})\})-\tau.
\end{equation}
If there exists $i\in\mathbb{N},i\geq 1$ such that
$$\max\{\met(a_i,a_{i+1}),\met(a_i,a_{i-1})\}=\met(a_i,a_{i+1}),$$
then, by inequality (\ref{big-manipulation}) we obtain
$$F(\met(a_i,a_{i+1}))\leq F(\met(a_i,a_{i+1}))-\tau.$$
Which is not true.
Hence, we must have
\begin{equation}
\label{u-equation}
    \max\left\{\met(a_i,a_{i-1}),\met(a_i,a_{i+1})\right\}=\met(a_i,a_{i-1}).
\end{equation}
The inequality (\ref{big-manipulation}) and the equation (\ref{u-equation}) give
\begin{equation}
\label{important-inequality-1}
\begin{split}
    F(\met(a_i,a_{i+1})) &\leq F(\met(a_{i-1},a_i))-\tau\\
    &\leq  F(\met(a_{i-2},a_{i-1}))-2\tau\\
    &\leq F(\met(a_{i-3},a_{i-2}))-3\tau\\
    & \ldots\ldots\ldots\ldots\ldots\ldots\\
    &\leq F(\met(a_0,a_1))-i\tau
\end{split}
\end{equation}
for all $i\in\mathbb{N},i\geq 1$.
The inequality (\ref{important-inequality-1}) implies that
$$\lim_{i\to\infty}F(\met(a_i,a_{i+1}))=-\infty.$$
Using (F2) it follows that
\begin{equation*}
    \label{d_i-to-0}
    \lim_{i\to\infty}\met(a_i,a_{i+1})=0.
\end{equation*}
By (F3) there exists $k\in(0,1)$ such that
\begin{equation}
    \label{d_i-k-d-to-0}
    \lim_{i\to\infty} (\met(a_i,a_{i+1}))^kF(\met(a_i,a_{i+1}))=0.
\end{equation}
Note that
\begin{equation}
    \label{sandwitch}
    \begin{split}
        &\met(a_i,a_{i+1})^kF(\met(a_0,a_1))-\met(a_i,a_{i+1})^kF(\met(a_i,a_{i+1}))\\
        &\geq \met(a_i,a_{i+1})^kF(\met(a_0,a_1))-\met(a_i,a_{i+1})^k(F(\met(a_0,a_1))-i\tau)\\
        &= i\met(a_i,a_{i+1})^k\tau\geq 0.
    \end{split}
\end{equation}
Using (\ref{d_i-k-d-to-0}) and (\ref{sandwitch}) we obtain
\begin{equation*}
    \lim_{i\to\infty}i(\met(a_i,a_{i+1}))^k=0.
\end{equation*}
Therefore, we can find $i_0\in\mathbb{N}$ such that
\begin{equation*}
    \begin{split}
        &i(\met(a_i,a_{i+1}))^k\leq 1 \text{ for all } i\geq i_0\\
        &\Rightarrow \met(a_i,a_{i+1})\leq \frac{1}{i^{1/k}}\text{ for all } i\geq i_0.
    \end{split}
\end{equation*}
Let $j,i\in\mathbb{N}$ be such that $j>i\geq i_0$. Therefore,
\begin{equation*}
    \begin{split}
        \met(a_j,a_i)&\leq \met(a_j,a_{j-1})+\ldots + \met(a_{i+1},a_i)\\
        &< \sum_{l=i}^{\infty}\met(a_l,a_{l+1})\leq \sum_{l=i}^{\infty}\frac{1}{l^{1/k}}.
    \end{split}
\end{equation*}
Since the series $\sum_{l=1}^{\infty}\frac{1}{l^{1/k}}$ converges, it follows that $\{a_i\}_{i\in\mathbb{N}}$ is a Cauchy sequence. But $S_1^0$, being a closed subset of a complete metric space, is itself a complete metric space. Hence, there exists $a\in S_1^0$ such that
\begin{equation*}
    \lim_{i\to\infty}a_i=a.
\end{equation*}
\end{proof}
The Example \ref{ex} shows that hypothesis (a) of Lemma \ref{sequence-lemma} cannot be dropped.
\begin{example}
\label{cartesian-example}
Let $(\metset,\met)$ be a metric space. Consider the cartesian product $\metset'=\mathbb{R}\times \metset$ under the metric $$\met'((u,v),(t,w))=\sqrt{\left |u-t\right |^2+\met(v,w)^2}$$
for all $u,t\in \mathbb{R}$ and $v,w\in \metset$. Define, 
$$S_1=\{(-1,x):x\in \metset\}$$
$$S_2=\{(1,x):x\in \metset\}.$$
Let $\phi',\psi':\metset\to \metset$ be mappings. We define $\phi,\psi:S_1\to S_2$ as follows:
$$\phi(-1,x)=(1,\phi'x)$$
$$\psi(-1,x)=(1,\psi'x).$$
It is easy to check that $\phi$ and $\psi$ commute proximally if and only if $\phi'$ and $\psi'$ commute.
\end{example}
\noindent
\begin{minipage}{.5\linewidth}
\begin{example}
\label{circle-example}
Let $\metset=\mathbb{C}$ under the metric $\met(z,w)=|z-w|$. Let $0<a<b$ be real numbers. Define
$$S_1=\{z\in\mathbb{C}:|z|=a\}$$
and
$$S_2=\{z\in\mathbb{C}:|z|=b\}.$$
For any point $z$ on $S_1$, draw a diameter of $S_2$ passing through $z$ as shown in Figure \ref{circ}. The diameter cuts the circle $S_2$ at two points. Denote the point nearer to $z$ by $\phi z$ and the other point by $\psi z$. It is easy to see that the mappings $\phi,\psi:S_1\to S_2$ commute proximally.
\end{example}
\end{minipage}\hfill
\begin{minipage}{.45\linewidth}
\begin{center}
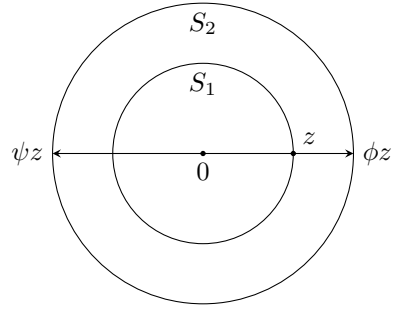

\begin{tikzpicture}[scale=0.40]
\draw (0,0) circle (5cm);
\draw (0,0) circle (3cm);
\draw [->,>=stealth](3,0) -- (5,0);
\draw [->,>=stealth](3,0) -- (-5,0);
\draw[fill=black] (0,0) circle (2pt);
\draw[fill=black] (3,0) circle (2pt);
\node (A) at (0,3) [below]{$S_1$};
\node (B) at (0,5) [below]{$S_2$};
\node (o) at (0,0) [below]{$0$};
\node (z) at (3,0) [above right]{$z$};
\node (fz) at (5,0) [right]{$\phi z$};
\node (gz) at (-5,0) [left]{$\psi z$};
\end{tikzpicture}
\captionof{figure}{Geometrical description of the mappings $\phi$ and $\psi$}
\label{circ}
\end{center}
\end{minipage}
\begin{example}\label{ex}
Let $\metset, \met, S_1, S_2, \phi, \psi$ be as in Example \ref{circle-example}. Clearly, in this case $S_1=S_1^0$ and $S_2=S_2^0$. It was discussed in Example \ref{circle-example} that $\phi$ and $\psi$ commute proximally. It is easy to check that the hypotheses (b) and (c) of Lemma \ref{sequence-lemma} are also satisfied. To see that the hypothesis (a) is not satisfied, let $z_1$ and $z_2$ be points in $S_1$ that lie on the opposite ends of a diameter of $S_1$. Take $a_1=b_2=z_1$ and $a_2=b_1=z_2$. We have
$$
\left.\begin{array}{cc}
     &  \met(a_1,\phi z_1)=\met(a_2,\phi z_2)=\met(S_1,S_2)\\
     &  \met(b_1,\psi z_1)=\met(b_2,\psi z_2)=\met(S_1,S_2)\\
     &  \quad\quad a_1\neq a_2.
\end{array}\right\}
$$
Let if possible $\phi$ be $F$-weakly dominated by $\psi$ proximally for some $F\in\mathcal{F}$. Then, we can find some $\tau>0$ such that
\begin{equation*}
    \begin{split}
    \tau+F(\met(a_1,a_2)) & \leq  F(M_\met(a_1,a_2,b_1,b_2))\\
    &= F(\met(a_1,a_2)).
   \end{split}
\end{equation*}
Which is a contradiction. 

Let, if possible, the conclusion of Lemma \ref{sequence-lemma} be satisfied. Then, there exists a sequence $\{a_i\}_{i\in\mathbb{N}}$ in $S_1$ such that
$$\phi a_i=\psi a_{i+1}$$
for all $i\in\mathbb{N}$. Since $\phi$ and $\psi$ cannot satisfy the condition $\phi a=\psi a$ for any $a\in S_1$, hence, $a_i\neq a_{i+1}$ for all $i\in\mathbb{N}$. Thus, the sequence $\{a_i\}_{i\in\mathbb{N}}$ must converge to some $a\in S_1$. But, since $\phi$ and $\psi$ are continuous, this implies that $\phi a=\psi a$, which is again a contradiction.
\end{example}
\begin{lemma}
\label{coincidence-lemma}
Let $S_1$ and $S_2$ be nonempty subsets of a metric space $(\metset,\met)$. Let $F\in\mathcal{F}$ and the mappings $\phi,\psi:S_1\to S_2$ be a pair of proximally commuting mappings which satisfy the following conditions:
\begin{enumerate}[(a)]
    \item $\psi$ is such that it $F$-weakly dominates $\phi$ proximally;
    \item $\phi(S_1^0)\subset S_2^0$;
    \item there exists $x\in S_1^0$ such that $\phi x=\psi x$
\end{enumerate}
then, $\phi$ and $\psi$ have a unique common best proximity point.
\end{lemma}
\begin{proof}
By hypothesis (b), hypothesis (c) and the definition of $S_2^0$, there exists some $a\in S_1^0$ such that
\begin{equation}
\label{v-equation}
    \left. \begin{array}{cc}
     &  \met(a,\phi x)=\met(S_1,S_2)\\
     &  \met(a,\psi x)=\met(S_1,S_2).
\end{array}\right\}
\end{equation}
The proximal commutativity of $\phi$ and $\psi$ along with the system (\ref{v-equation}) implies
$$\phi a=\psi a.$$
Again, as $\phi(S_1^0)\subset S_2^0$, we can find $b\in S_1^0$ such that
\begin{equation}
\label{w-euation}
    \left. \begin{array}{cc}
     &  \met(b,\phi a)=\met(S_1,S_2)\\
     &  \met(b,\psi a)=\met(S_1,S_2).
\end{array}\right\}
\end{equation}
We claim that $b=a$. Suppose not, hence from (\ref{v-equation}) and (\ref{w-euation}) we obtain the following system
\begin{equation}
    \label{v-w-equality-system}
    \left. \begin{array}{cc}
     &  \met(a,\phi x)=\met(b,\phi a)=\met(S_1,S_2)\\
     &  \met(a,\psi x)=\met(b,\psi a)=\met(S_1,S_2)\\
     &  \quad \quad a\neq b.
\end{array}\right\}
\end{equation}
System (\ref{v-w-equality-system}) along with hypothesis (a) implies that there exists some $\tau>0$ such that
$$ \tau+F(\met(a,b))\leq F(M_\met(a,b,a,b))=F(\met(a,b))
$$
which is not true. So $a=b$. Hence, equations (\ref{w-euation}) give $$\met(a,\phi a)=\met(a,\psi a)=\met(S_1,S_2).$$ Thus, $a$ is a common best proximity point of $\phi$ and $\psi$.

For uniqueness, let $a_i$(i=1,2) be common best proximity points of $\phi$ and $\psi$ such that $a_1\neq a_2$. Thus, we obtain the following system of equations:
\begin{equation}
\label{uniqueness-system}
    \left. \begin{array}{cc}
     &  \met(a_1,\phi a_1)=\met(a_2,\phi a_2)=\met(S_1,S_2)\\
     &  \met(a_1,\psi a_1)=\met(a_2,\psi a_2)=\met(S_1,S_2)\\
     &  \quad \quad a_1\neq a_2.
\end{array}\right\}
\end{equation}
System (\ref{uniqueness-system}) along with hypothesis (a) implies that
$$\tau+F(\met(a_1,a_2))\leq F(M_\met(a_1,a_2,a_1,a_2))=F(\met(a_1,a_2))$$
which is not true. Hence, we must have $a_1=a_2$.
\end{proof}
For the mappings $\phi,\psi:S_1\to S_2$, a point $x\in S_1$ which satisfies $\phi x=\psi x$ is called a coincidence point of $\phi$ and $\psi$. There is a vast amount of literature available on coincidence point theorems. If we add the hypotheses of those coincidence point theorems in Lemma \ref{coincidence-lemma}, we obtain a corresponding common best proximity point theorem. However, the hypotheses of Lemma \ref{coincidence-lemma} are still quite strong, so it may happen that no interesting examples of the resulting theorems exist. One may think that we can drop some hypothesis from it to obtain a more general result. The following example shows that at least hypothesis (a) of Lemma \ref{coincidence-lemma} cannot be dropped.
\begin{example}
Consider $\mathbb{R}$ under the usual Euclidean metric $d$. Define $\phi,\psi:\mathbb{R}\to\mathbb{R}$ as
$$\phi x=\begin{cases}-\frac{1}{x}& x\neq 0\\\quad 1& x=0\end{cases}$$
$$\psi x=\begin{cases}-\frac{1}{x^3}&x\neq 0\\\quad 1 & x=0\end{cases}$$
for all $x\in\mathbb{R}$.
With $S_1=S_2=\mathbb{R}$, it is easy to see that $\phi$ and $\psi$ commute proximally, and they also satisfy the hypotheses (b) and (c) of Lemma \ref{coincidence-lemma}. Note that $\phi$ and $\psi$ does not have any common best proximity point (i.e. no common fixed point in this case).

We claim that the hypothesis (a) is not satisfied. Suppose that the condition is satisfied. Note that $\phi(0)\neq \phi(1)$. If $\phi$ is $F$-weakly dominated by $\psi$ for some $F\in\mathcal{F}$, then by Definition \ref{f-g-x-propostion} there exists $\tau>0$ such that
\begin{equation*}
    \begin{split}
        &\tau+F(|\phi(0)-\phi(1)|)\leq F(M_\met(\phi(0),\phi(1),\psi(0),\psi(1)))\\
        &\Rightarrow \tau + F(2)\leq F(2)
    \end{split}
\end{equation*}
which is a contradiction.
\end{example}
It may be interesting to look for other conditions on $\phi$ and $\psi$ which guarantee that the existence of a coincidence point implies the existence of a common best proximity point.

By observing the statements of Lemma \ref{sequence-lemma} and Lemma \ref{coincidence-lemma}, one may notice that if $\phi,\psi:S_1\to S_2$ satisfy all the hypotheses of Lemma \ref{sequence-lemma} along with some additional hypotheses on $\phi$ and $\psi$ which guarantee that the point $a$ given by Lemma \ref{sequence-lemma} satisfies $\phi a=\psi a$, then by Lemma \ref{coincidence-lemma}, $\phi$ and $\psi$ will have a unique common best proximity point. One such hypothesis is to assume that $\phi{_{|S_1^0}}$ is continuous and $\psi{_{|S_1^0}}$ is closed.
\begin{theorem}
\label{closed-continuous}
Let $F\in\mathcal{F}$ and $(\metset,\met)$ be a complete metric space. Consider any two nonempty subsets $S_1$ and $S_2$ of $\metset $ with $S_1^0$ as nonempty and closed. Let $\phi,\psi:S_1\to S_2$ be a pair of proximally commuting mappings which satisfy the following conditions:
\begin{enumerate}[(a)]
    \item $\psi$ is such that it $F$-weakly dominates $\phi$ proximally;
    \item $\phi{_{|S_1^0}}$ is continuous and $\psi{_{|S_1^0}}$ is closed;
    \item $\phi(S_1^0)\subset S_2^0$;
    \item $\phi(S_1^0)\subset \psi(S_1^0)$
\end{enumerate}
then, $\phi$ and $\psi$ have a unique common best proximity point.
\end{theorem}
\begin{proof}
By Lemma \ref{sequence-lemma}, we can find a sequence $\{a_i\}_{i\in\mathbb{N}}$ in $S_1^0$ which satisfies
\begin{equation}
\label{u_i-imp-relation}
    \phi a_i=\psi a_{i+1}
\end{equation}
for all $i\in\mathbb{N}$. If $$a_i=a_{i+1}=a$$ for some $i\in\mathbb{N}$, then from (\ref{u_i-imp-relation}) we get $$\phi a=\psi a.$$ Assume now that $a_i\neq a_{i+1}$ for all $i\in\mathbb{N}$. Then, by Lemma \ref{sequence-lemma}, there exists $a\in S_1^0$ such that the sequence $\{a_i\}_{i\in\mathbb{N}}$ converges to $a$.

Since $\phi{_{|S_1^0}}$ is continuous, we obtain
$$\lim_{i\to\infty}\phi a_i=\lim_{i\to\infty}\psi a_{i+1}=\phi a\in \phi(S_1^0)\subset \psi(S_1^0).$$
Since $\psi{_{|S_1^0}}$ is closed we must have $$\psi a=\phi a.$$ Therefore, by Lemma \ref{coincidence-lemma}, $\phi$ and $\psi$ have a unique common best proximity point.
\end{proof}
Another condition would be to assume that $\psi{_{|S_1^0}}$ is injective (i.e. $\psi x_1=\psi x_2\Rightarrow x_1=x_2$ for all $x_1,x_2\in S_1^0$) and the mapping $~{\xi:S_1^0\to S_1^0}$ given by $\xi=(\psi{_{|S_1^0}})^{-1}\circ \phi{_{|S_1^0}}$ is closed.
\begin{theorem}
Let $F\in\mathcal{F}$ and $(\metset,\met)$ be a complete metric space. Consider any two nonempty subsets $S_1$ and $S_2$ of $\metset $ with $S_1^0$ as nonempty and closed. Let $\phi,\psi:S_1\to S_2$ be a pair of proximally commuting mappings which satisfy the following conditions:
\begin{enumerate}[(a)]
    \item $\psi$ is such that it $F$-weakly dominates $\phi$ proximally;
    \item $\psi{_{|S_1^0}}$ is injective;
    \item the mapping $\xi:S_1^0\to S_1^0$ given by $\xi=(\psi{_{|S_1^0}})^{-1}\circ \phi{_{|S_1^0}}$ is closed;
    \item $\phi(S_1^0)\subset S_2^0$;
    \item $\phi(S_1^0)\subset \psi(S_1^0)$
\end{enumerate}
then, $\phi$ and $\psi$ have a unique common best proximity point.
\end{theorem}
\begin{proof}
We claim that there exists $a\in S_1^0$ such that $\phi a=\psi a$. Clearly, it is sufficient to show that the mapping $\xi:S_1^0\to S_1^0$ has a fixed point in $S_1^0$.

By Lemma \ref{sequence-lemma}, there exists a sequence $\{a_i\}_{i\in\mathbb{N}}$ in $S_1^0$ such that
\begin{equation}
\label{prep}
    \begin{split}
        & \phi a_i=\psi a_{i+1}\\
        &\Rightarrow \xi a_i=a_{i+1}
    \end{split}
\end{equation}
for all $i\in\mathbb{N}$. If $a_i=a_{i+1}=a$ for some $i\in\mathbb{N}$, then, by equation (\ref{prep}) $\xi a=a$. Suppose that $a_i\neq a_{i+1}$ for all $i\in\mathbb{N}$. By Lemma \ref{sequence-lemma}, there exists $a\in S_1^0$ such that
$$\lim_{i\to\infty}a_i=a.$$
Since $\xi$ is closed, from equation (\ref{prep}) we obtain $\xi a=a$.
\end{proof}

If $S_1=S_2=\metset$ in Theorem \ref{closed-continuous}, we obtain the following common fixed point theorem. Similar theorems were studied by Zhou \textit{et al.} in \cite{zhou2019coincidence}. It is also an improvement of Theorem 3.9 in \cite{batra2017common}.
\begin{theorem}
Consider an $F\in\mathcal{F}$ and a complete metric space $(\metset,\met)$. Let $\phi,\psi:\metset\to \metset$ be a pair of commuting mappings which satisfy the following conditions:
\begin{enumerate}[(a)]
    \item $\psi$ is such that it $F$-weakly dominates $\phi$;
    \item $\phi$ is continuous and $\psi$ is closed;
    \item $\phi(\metset)\subset \psi(\metset)$
\end{enumerate}
then, $\phi$ and $\psi$ have a unique common fixed point.
\end{theorem}

The following result due to R. Batra can be obtained from Theorem \ref{closed-continuous}:
\begin{cor}[\cite{batracommon}]
\label{sir-theorem}
Let $F\in\mathcal{F}$ and $(\metset,\met)$ be a complete metric space. Consider any two nonempty subsets $S_1$ and $S_2$ of $\metset $ with $S_1^0$ as non-empty and closed. Let $\phi,\psi:S_1\to S_2$ be a pair of proximally commuting mappings which satisfy the following conditions:
\begin{enumerate}[(a)]
    \item $\psi$ is such that it $F$-dominates $\phi$ proximally;
    \item $\psi$ and $\phi$ both are continuous;
    \item $\phi(S_1^0)\subset S_2^0$;
    \item $\phi(S_1^0)\subset \psi(S_1^0)$
\end{enumerate}
then, $\phi$ and $\psi$ have a unique common best proximity point.
\end{cor}
\begin{proof}
If $\phi$ and $\psi$ are continuous, then $\phi{_{|S_1^0}}$ and $\psi{_{|S_1^0}}$ are both continuous. By Remark \ref{continuous-are-closed}, $\psi{_{|S_1^0}}$ is closed. By Proposition \ref{proper-extension-proposition}, $\phi$ is $F$-weakly dominated by $\psi$ proximally. Hence, by Theorem \ref{closed-continuous}, $\phi$ and $\psi$ have a unique common best proximity point.
\end{proof}
The following example shows that Theorem \ref{closed-continuous} is a proper generalization of Corollary \ref{sir-theorem}.
\begin{example}\label{ex22}
Take $\metset=\{x\in\mathbb{R}:x\leq -1\text{ or }x\geq 1\}$ in Example \ref{cartesian-example} with the Euclidean metric $\met(x,y)=|x-y|$, so that $(\metset',\met')$ becomes a closed subset of the complete metric space $\mathbb{R}^2$ under the usual Euclidean metric of $\mathbb{R}^2$. Take $\phi',\psi':\metset\to \metset$ as follows:
$$\phi'x=\begin{cases}
3 & x\leq -1\\
5 & x\geq 1
\end{cases}$$
and
$$\psi'x=x.$$
Let $S_1, S_2, \phi$ and $\psi$ be as defined in Example \ref{cartesian-example}. Note that $\met(S_1,S_2)=2$. Clearly $S_1=S_1^0$ and $S_2=S_2^0$. Since $\phi'$ and $\psi'$ commute, hence, as we discussed in Example \ref{cartesian-example}, $\phi$ and $\psi$ commute proximally. Also, $~{\phi(S_1^0)\subset S_2^0}$ and $\phi(S_1^0)\subset \psi(S_1^0)$. Clearly, $\phi{_{|S_1^0}}$ and $\psi{_{|S_1^0}}$ are both continuous (and thus, $\psi{_{|S_1^0}}$ is closed).

First, we show that $\phi$ is $F_1$-weakly dominated by $\psi$ proximally, where $F_1$ is the natural logarithm function. To see this, suppose that $(-1,a_i),(-1,b_i),$ $(-1,x_i)\in S_1(i=1,2)$ be the points which satisfy the equations:
\begin{equation}
\label{D-equation-set}
    \left. \begin{array}{cc}
     &  \met'((-1,a_1),\phi(-1,x_1))=\met'((-1,a_2),\phi(-1,x_2))=2\\
     &  \met'((-1,b_1),\psi(-1,x_1))=\met'((-1,b_2),\psi(-1,x_2))=2\\
     &  \quad \quad (-1,a_1)\neq (-1,a_2).
\end{array}\right\}
\end{equation}
Equations (\ref{D-equation-set}) imply that
\begin{equation}
    \label{D-equation-implication}
    \left. \begin{array}{cc}
     &  a_1=\phi'b_1,a_2=\phi'b_2\\
     &  a_1\neq a_2.
\end{array}\right\}
\end{equation}
Clearly, equations(\ref{D-equation-implication}) imply that either $$b_1\leq -1,b_2\geq 1 \text{ or } b_2\leq -1, b_1\geq 1.$$
First, we assume that $b_1\leq -1$ and $b_2\geq 1$. It follows that
\begin{equation}
\label{D-bounded}
    \begin{split}
        &\frac{\met'((-1,a_1),(-1,a_2))}{M_{\met'}((-1,a_1),(-1,a_2),(-1,b_1),(-1,b_2))}\\
        &=\frac{\met(a_1,a_2)}{M_\met(a_1,a_2,b_1,b_2)}\\
        &\leq\frac{\met(a_1,a_2)}{\met(a_1,b_1)}=\frac{\met(\phi'b_1,\phi'b_2)}{\met(\phi'b_1,b_1)}\\
        &= \frac{|3-5|}{|3-b_1|}=\frac{2}{|3-b_1|}\leq \frac{2}{4}= \frac{1}{2}.
    \end{split}
\end{equation}
If $b_2\leq -1$ and $b_1\geq 1$, we can prove that the inequality (\ref{D-bounded}) still holds by using $M_\met(a_1,a_2,b_1,b_2)\geq \met(a_2,b_2)$. Inequality (\ref{D-bounded}) gives,
\begin{multline*}
\ln(2)+\ln(\met'((-1,a_1),(-1,a_2)))\\\leq \ln(M_{\met'}((-1,a_1),(-1,a_2),(-1,b_1),(-1,b_2))).
\end{multline*}
Hence, $\phi$ is $F_1$-weakly dominated by $\psi$ proximally with $\tau=\ln(2)$.
Now, we claim that $\phi$ is not $F$-dominated by $\psi$ proximally for any $F\in\mathcal{F}$. Note that $a_1=3, a_2=5,b_1=x_1=-1, b_2=x_2=1$ satisfy the system of equations (\ref{D-equation-set}). If $\phi$ is $F$ dominated by $\psi$ proximally for some $F\in\mathcal{F}$, then there must exist some $\tau>0$ such that
\begin{equation*}
\begin{split}
    &\tau+F(\met'((-1,3),(-1,5)))\leq F(\met'((-1,-1),(-1,1)))\\
    &\Rightarrow \tau+F(2)\leq F(2)\\
\end{split}
\end{equation*}
which is a contradiction.
Of course, the unique common best proximity point of $\phi$ and $\psi$, as guaranteed by Theorem \ref{closed-continuous} is $(-1,5)$.
\end{example}
\section{Acknowledgement}As a Ph.D. student, the first author is grateful to the University of Delhi and the Government of India for awarding him with a Non N.E.T. fellowship.   
\section{Conflict of interest} Authors declare that there is no conflict of interest regarding the publication of this article.
\section{Authors' contribution} Both authors contributed equally in writing this article. Both authors read and approved
the final manuscript.
\printbibliography
\end{document}